\documentclass[a4paper,leqno]{amsart}

\usepackage{amsmath}
\usepackage{amsthm}
\usepackage{amssymb}
\usepackage{amsfonts}
\usepackage[applemac]{inputenc}
\usepackage{mathrsfs}
\usepackage{pdfsync}
\usepackage{graphicx}
\usepackage{mathabx}
\usepackage{color}
\def\R{{\mathbb R}}% real numbers
\def\N{{\mathbb N}}% nonnegative integers
\def\Z{{\mathbb Z}}% integers
\def\le{\leqslant}% lessorequal
\def\ge{\geqslant}%greaterorequal

\theoremstyle{plain}
\newtheorem{theorem}{Theorem}[section]
\newtheorem{lemma}[theorem]{Lemma}

\newtheorem{proposition}[theorem]{Proposition}
\newtheorem{hyp}{Assumption}[section]

\theoremstyle{definition}

\newtheorem{remark}[theorem]{Remark}
\newtheorem*{remark*}{Remark}

\numberwithin{equation}{section}

\begin{document}

\title[Nonlinear bound states with prescribed angular momentum]
{Nonlinear bound states with prescribed angular momentum}
\author[I. Nenciu, X. Shen, C. Sparber]{Irina Nenciu, Xiaoan Shen, Christof Sparber}

\address[I.~Nenciu]
{Department of Mathematics, Statistics, and Computer Science, M/C 249, University of Illinois at Chicago, 851 S. Morgan Street, Chicago, IL 60607, USA \textit{and} Institute of Mathematics ``Simion Stoilow''
     of the Romanian Academy\\ 21, Calea Grivi\c tei\\010702-Bucharest, Sector 1\\Romania}
\email{nenciu@uic.edu}

\address[X. Shen]
{Department of Mathematics, Statistics, and Computer Science, M/C 249, University of Illinois at Chicago, 851 S. Morgan Street, Chicago, IL 60607, USA}
\email{xshen30@uic.edu}

\address[C.~Sparber]
{Department of Mathematics, Statistics, and Computer Science, M/C 249, University of Illinois at Chicago, 851 S. Morgan Street, Chicago, IL 60607, USA}
\email{sparber@uic.edu}

\begin{abstract}
We prove the existence of a class of orbitally stable bound state solutions to nonlinear Schr\"odinger equations with super-quadratic confinement in two and three spatial dimensions. These solutions are given by 
time-dependent rotations of a non-radially symmetric spatial profile which in itself is obtained via a doubly constrained energy minimization. One of the two constraints imposed is the total mass, 
while the other is given by the expectation value of the angular momentum around the $z$-axis. Our approach also allows for a new description of the set of 
minimizers subject to only a single mass constraint.
\end{abstract}

\date{\today}


\thanks{This publication is supported by the MPS Simons foundation through awards no. 851720 and no. 709025}
\maketitle

%%%%%%%%%%%%%%%%%%%%%%%%%%%%%%%%%%%
%%%%%%%%%%%%%%%%%%%%%%%%%%%%%%%%%%%

\section{Introduction}
\label{sec:intro}

We consider, for $(t,x) \in \R \times \R^d$, with $d=2$, or $d=3$, the nonlinear Schr\"odinger equation (NLS)
\begin{equation}\label{nls}
	i \partial_{t} u = H u  + \lambda |u|^{2\sigma}u, \quad u_{\vert{t=0}} = u_{0},
\end{equation}
where $\lambda \in \R$, $\sigma>0$, and
\begin{equation*}\label{ham}
H = -\frac{1}{2} \Delta +V(x),
\end{equation*}
the linear part of the Hamiltonian. 
Here, $V$ is a smooth confining potential, which is assumed to grow super-quadratically at infinity. More precisely, we impose:
\begin{hyp}\label{ass1} 
The potential $V\in C^\infty(\R^d;\R)$ is assumed to be radially symmetric and confining, i.e. $V(x) \to +\infty$ as $|x|\to \infty$. Moreover, there exists $k>2$ and $R>0$, 
such that for $|x|>R$:
\[
c_\alpha \langle x \rangle ^{k-|\alpha|} \le |\partial^\alpha V(x)| \le C_\alpha \langle x \rangle ^{k-|\alpha|}, \quad \forall \, \alpha \in \N^d,
\]
where $ \langle x \rangle=(1+|x|^2)^{1/2}$ and $c_\alpha, C_\alpha$ are positive constants.
\end{hyp}

For simplicity, we assume that 
\[
V(x)\ge 0\quad \text{ for all }x\in \R^d\,.
\] 
Indeed, Assumption \ref{ass1} implies that $V$ is bounded below and hence $V\ge 0$ can always be achieved by a simple gauge transform.
\begin{remark} A typical example for an admissible potential is given by $V(x) = |x| ^{k}$, with $k>2$. The limiting case of a quadratic potential where $k=2$ is 
excluded for reasons which will become clear below.
\end{remark}
Assumption \ref{ass1} implies that the operator $H$ defined on $C_0^\infty(\R^d)$ is essentially self-adjoint 
on $L^2(\R^d)$, giving rise to the linear Schr\"odinger group $\big(e^{-i tH}\big)_{t\in \R}$, such that
\[
 e^{-i tH} :  L^2(\R^d)\to L^2(\R^d) \ \text{unitary}.
\]
The nonlinear dynamics given by \eqref{nls} (formally) conserves the mass 
\begin{equation}\label{mass}
M(u) : = \| u \|_{L^2(\R^d)}^2\,,
\end{equation}
and the total energy 
\begin{equation}\label{energy}
\begin{split}
E(u) := &\, \int_{\R^d} \frac{1}{2} |\nabla u|^2 + V(x) |u|^2 + \frac{\lambda}{\sigma +1}  | u|^{2\sigma+2} \, dx \\=
& \big\|H^{1/2}u\big\|^2_{L^{2}(\R^d)}+\frac{\lambda}{\sigma +1} \| u \|^{2\sigma+2}_{L^{2\sigma+2}(\R^d)}.
\end{split}
\end{equation}
Another important physical quantity is the mean angular momentum of $u$ around a given rotation axis in $\R^3$. 

To fix ideas, and without loss of generality, we 
denote the coordinates of a point in $\R^3$ as $x=(x_1,x_2,z)$ and
assume that the rotation axis is the $z$-axis. 
In this 
case, the mean angular momentum of $u$ is given by
\[
L (u): = {\langle u,  L_zu \rangle}_{L^2(\R^d)}\,,
\]
where $L_z$ is the third component of the quantum mechanical angular momentum operator $\mathbb L=-i x\wedge \nabla$ , i.e.
\begin{equation*}
L_z u= -i \big(x_1 \partial_{x_2} u- x_2\partial_{x_1}u\big)\,.
\end{equation*}
Note that in cylindrical coordinates $(r,z, \varphi)$ in $\R^3$, we simply have
\begin{equation}\label{eq:Lphi}
L_z u = -i \partial_\varphi u.
\end{equation}
In $\R^2$ we use the standard convention of simply setting $x=(x_1,x_2)$, and $L$, $L_z$ remain as above.

A simple computation (see \cite{AMS}) shows that the time-evolution of $L(u)$ under the flow of \eqref{nls} satisfies: 
\[
L(u(t, \cdot)) + i \int_0^t \int_{\R^d} |u(\tau,x)|^2 L_z V(x) \, dx \, d\tau = L(u_0).
\]
Thus, in the case where $V$ is axis-symmetric, i.e. $L_z V\equiv 0$, the dynamics of \eqref{nls} also satisfies the angular momentum conservation law
\begin{equation*}\label{Lmomentum}
L(u(t, \cdot)) =L(u_0),\quad \forall t\in \R.
\end{equation*}
In particular, this is true in our situation, in view of \eqref{eq:Lphi} and the assumption that $V$ is radially symmetric.
\smallskip

To make these formal computations rigorous, let $s\ge 0$ and consider the family of natural energy spaces
\[
\mathcal H^s := \big\{ f\in L^2(\R^d) \, : \, \|  f\|^2_{\mathcal H^s} \equiv \| f \|^2_{L^2(\R^d)}+ \| H^{s/2} f\|^2_{L^2(\R^d)} <\infty \big \}\,.
\]
It is shown in \cite{YZ} that we have the norm equivalence
\[
\|  f\|^2_{\mathcal H^s} \simeq \|  f\|^2_{H^s(\R^d)} +\|  V^{s/2} f \|^2_{L^2(\R^d)}\, ,
\]
where $H^s(\R^d)$ denotes the usual $L^2$-based Sobolev space of order $s\ge 0$. Sobolev's Embedding theorem guarantees that
\[
H^{1}(\R^d)\hookrightarrow L^{2\sigma+2}\ \text{provided that $\sigma<\frac{2}{(d-2)_+}$}\,.
\] 
In particular, for $d=2$ we have $\sigma < \infty$. Thus, under this restriction on $\sigma>0$, both $M(u)$ and $E(u)$ are well defined functionals 
on $\mathcal H^1$. 
By employing Young's inequality, one also sees that $L(u)$ is well-defined. Indeed,
\begin{align*}
\left|\langle u, L_z u\rangle  \right|  & \le \|x u \|_{L^{2}(\R^d)}\|\nabla u \|_{L^{2}(\R^d)}  \lesssim \| xu\|^{2}_{L^{2}(\R^d)} +  \|\nabla u \|^{2}_{L^{2}(\R^d)}\\
&=\int_{|x|\leq R}|xu|^2\, dx+\int_{|x|>R} |xu|^2\, dx+\|\nabla u\|^{2}_{L^{2}(\R^d)} ,
\end{align*} 
and we can estimate 
\begin{align*}
\left|\langle u, L_z u\rangle  \right|  & \lesssim R^2\int_{|x|\leq R}|u|^2\, dx+\int_{|x|>R}V^s(x)|u|^2\, dx+\|\nabla u\|^{2}_{L^{2}(\R^d)}\\& \leq R^2\|u\|^{2}_{L^{2}(\R^{d})}+
 \|V^{s/2}u\|^{2}_{L^{2}(\R^{d})}+\|\nabla u\|^{2}_{L^{2}(\R^d)}\lesssim \|  u\|^2_{\mathcal H^s}\,.&
\end{align*}
Here $R$ is chosen as in Assumption \ref{ass1}, which guarantees that $|x|^2\lesssim V(x)^s$ for $|x|>R$ and $s\geq 1$.

Using space-time Strichartz estimates established in \cite{YZ}, together with the conservation laws above, 
then yields the following well-posedness result:

\begin{proposition}[\cite{Ca}] \label{prop:GWP} 
Let $d=2,3$, and $V$ satisfy Assumption \ref{ass1}. Let $s\ge 1$, $u_0\in \mathcal H^s$, and assume that one of the following holds:
\begin{itemize} 
\item[(i)] $\lambda\le 0$ and $\sigma < \frac{2}{d}$, or
\item[(ii)] $\lambda > 0$ and $\sigma < \frac{k+2}{k(d-2)_+}$. 
\end{itemize} 
Then there exists a unique global 
solution $u\in C(\R; \mathcal H^s)$ to \eqref{nls}, depending continuously on $u_0$, and satisfying the conservation laws of mass, energy, and angular momentum.
\end{proposition}

Note that for $d=3$, the upper bound in (ii) becomes $\sigma < 1+\frac{2}{k}$ for some $k>2$ given by Assumption \ref{ass1}. 
This consequently allows for the inclusion of nonlinearities slightly larger than the cubic one ($\sigma =1$) and simultaneously is
within the range of the Sobolev-imbedding $H^1(\R^3)\hookrightarrow L^{2\sigma+2}(\R^3)$ with $\sigma <2$.

\begin{hyp} \label{ass2} We assume that $\lambda$ and $\sigma$ satisfy one of {\rm (i)} or {\rm (ii)} above.\end{hyp}

Now, let $\Omega \in \R$ be a given angular velocity and recall that $\Omega L_z$ is the generator of time-dependent rotations around the 
$z$-axis, in the sense that for any $f \in L^2(\R^d)$:
\[
e^{it \Omega  L_z } f(x)= f\left(e^{-t\Theta }x \right),
\]
where $\Theta$ is the skew symmetric matrix given by
\[
\Theta=
  \begin{pmatrix}
    0 & \Omega \\
    -\Omega & 0 
  \end{pmatrix}\ \text{for $d=2$, and } 
  \Theta=
  \begin{pmatrix}
	0 & \Omega & 0\\
	-\Omega & 0 & 0\\
	0 & 0 & 0
	\end{pmatrix} \ \text{for $d=3$.}
\]
Clearly, $\big( e^{-it \Omega L_z} \big)_{t\in \R}$ is a family of unitary operators
\[
e^{it \Omega L_z}: \mathcal H^s \to \mathcal H^s.
\]
 Let $u(t, \cdot) \in \mathcal H^1$ be a global solution to \eqref{nls}, as guaranteed by Proposition \ref{prop:GWP}, and define a new unknown
\begin{equation}\label{eq:vu}
v(t,x) := e^{it \Omega  L_z } u(t,x) = u\left(t, e^{-t\Theta }x \right).
\end{equation}
A straightforward computation (cf. \cite{AMS}) then shows that $v$ satisfies the following equation 
\begin{equation}\label{vnls}
	i\partial_{t} v=-\frac{1}{2}\Delta v +\lambda|v |^{2\sigma} v+ V(x) v-\Omega L_z v, \quad v_{\vert{t=0}} = u_{0}.
\end{equation}
Here, we use the fact that $V$ is radially symmetric, and hence it commutes with the action of $e^{-it \Omega  L_z }$. In particular, $L\big(v(t,\cdot)\big)=L\big(u(t, \cdot)\big)$ for all $t\in \R$.

The NLS-type equation \eqref{vnls} appears in the mean-field description of {\it rotating Bose-Einstein condensates}, see, e.g., \cite{AMS, ANS, RS1}. In particular, time-periodic solutions of the form
\[
v(t,x) = \phi(x) e^{-i\omega t},
\]
satisfy the stationary NLS equation with rotation:
\begin{equation}\label{statnls}
	H\phi+\lambda|\phi|^{2\sigma}\phi=\omega\phi+\Omega L_z\phi .
\end{equation}
The latter is usually considered to be the Euler-Lagrange equation of the associated {\it Gross-Pitaevskii energy functional} with additional rotation term, i.e.
\begin{equation}\label{GPenergy}
E_\Omega(u) =\int_{\R^d} \frac{1}{2} |\nabla u|^2 + V(x) |u|^2 + \frac{\lambda}{\sigma +1}  | u|^{2\sigma+2} \, dx - \Omega L(u).
\end{equation}
Existence of physical {\it ground states} $\phi$ satisfying \eqref{statnls} can be proven 
by minimizing this $E_\Omega$ subject to a mass constraint. More precisely, for any given $m>0$, we denote the sphere
\[
\mathcal S_m := \big \{ u\in \mathcal{H}^1 \, : \, M(u) = m \big\}\,,
\]
where $M(u)$ is defined in \eqref{mass}, and consider the following minimization problem:
\begin{equation}\label{eq:mini1}
e_\Omega(m):=\inf _{u\in \mathcal S_m} E_\Omega(u)\,.
\end{equation}
If this infimum is achieved, i.e. if there exists $\phi \in \mathcal S_m$ such that $e_\Omega(m) = E_\Omega(\phi)$, 
then \eqref{statnls} can be seen as an equation for critical points satisfying the Lagrange condition
\begin{equation}\label{Lag1}
dE_\Omega(\phi) = \omega dM(\phi).
\end{equation}
Here, $\omega\in \R$ is the Lagrange multiplier associated to the mass constraint, usually called the {\it chemical potential}. Note that if $\phi \in \mathcal S_m$ 
is a minimizer, then so is $e^{i \theta} \phi$ for $\theta \in \R$, due to the gauge invariance of the mass $M$ and energy $E_\Omega$.

This approach has been successfully implemented in the seminal works \cite{RS1, RS2}, where it was proven that the set of ground states
\[
\mathcal G_{m, \Omega} := \big \{ \phi\in \mathcal S_{m}: E_\Omega(\phi)=e_\Omega(m) \big \}\,\not = \emptyset,
\] 
and several qualitative properties of these minimizers were established. Later studies in the same direction can be found 
in, e.g., \cite{ANS, BCPY, CRY}. All of these seek to give a mathematical description of {\it quantum vortices} which are known to appear in Bose-Einstein condensates under rotational forcing, cf. \cite{Co, Fe}. 
\begin{remark}
Note that, in the case of {\it quadratic} confining potentials $V$, an additional smallness condition on $|\Omega|$ is required to 
guarantee that $E_\Omega(u)$ is bounded below, cf. \cite{ANS, BC, RS1, RS2}. In the case of super-quadratic potentials, however, no such requirement arises (see the introduction of 
\cite{RS1}, in which the case of polynomially bounded, confining potentials $V$ is discussed). In particular, this 
allows for the appearance of a multiply quantized {\it giant vortex} provided $|\Omega|$ is sufficiently large, cf. \cite{CPRY}.
\end{remark}
A drawback of the approach outlined above is that one does not know how much vorticity the obtained minimizer $\phi$ carries. In particular, the numerical value $l\in \R$ of the mean angular momentum
\[
L(\phi)=\langle \phi , L_z \phi \rangle_{L^2(\R^d)}=l
\]
remains unknown.
Indeed, {\it a-priori} one can not even exclude the possibility that the minimizer $\phi$ is radially symmetric, in which case $L_z\phi\equiv 0$ 
and the last term on the r.h.s. of \eqref{statnls} simply vanishes. That this is not the case, in general, has 
been proved in \cite{RS1} in dimension $d=2$, and in \cite{RS2} for $d=3$ (see also \cite{BC, BQZ} for numerical simulations). 
Because a nonzero value of $L(\phi)$ signifies the onset of vortex nucleation in Bose-Einstein experiments, one might therefore ask the following question: 
For any given mass $m>0$, is it possible to guarantee the existence of stationary profiles $\phi$, which exhibit a certain predetermined mean angular momentum $0\not = l = L(\phi)$? 

We shall answer this question affirmatively below, by reinterpreting \eqref{statnls} as the Euler-Lagrange equation for a doubly constrained minimization problem. 
More precisely, instead of interpreting the stationary equation \eqref{statnls} via \eqref{Lag1}, we 
shall consider the following minimization problem:
\begin{equation}\label{eq:mini2}
e(m,l):=\inf \big\{E(u):u\in \mathcal{H}^1, M(u)=m , L(u)=l\big \}\,,
\end{equation}
where $m>0$ and $l\in \R$ are given constants, and $E(u)$ is the original energy functional defined in \eqref{energy}. 
In other words, we propose a 
different point of view to the problem of finding energy minimizers for rotating Bose-Einstein condensates within the framework of 
Gross-Pitaevskii theory. Instead of trying to minimize the rotating 
energy functional $E_\Omega$ subject to a single mass constraint, we 
consider the original energy $E$ subject to {\it two} constraints, one for the mass $M(u)=m$ and one for the (mean) 
angular momentum $L(u) = l$. As far as we know, such a doubly constrained 
minimization problem has never been studied before. 

In contrast to the previous minimization problem, the constraining set 
\[
\mathcal C_{m,l}:=\big \{\phi\in \mathcal{H}^1: M(\phi)=m, L(\phi)=l \big\}
\]
is no longer a sphere, but, as shown in Lemma \ref{lem:nonempty}, can be seen to contain 
a set which is isometrically isomorphic to the direct sum of two infinite dimensional 
spheres in $\ell^2(\Z)$.
Assuming for the moment that a minimizer exists, i.e. there is $\phi \in \mathcal C_{m,l}$ such that 
$e(m, l) = E(\phi)$, 
the stationary equation \eqref{statnls} is then obtained as the following Euler-Lagrange condition:
\[
dE(\phi) = \omega dM(\phi) + \Omega dL(\phi).
\]
Here $\Omega\in \R$ is a second Lagrange multiplier associated to the angular momentum constraint $L(\phi) = l$. Clearly, if $l\not =0$, any such doubly constrained minimizer $\phi$ 
{\it cannot be radially symmetric} by construction. In turn, we do not know whether  uniqueness of minimizers (up to phase conjugation) holds in general. 

\smallskip

Our first main result can then be stated as follows:

\begin{theorem}[Existence of minimizers]\label{thm:main}
Suppose that Assumptions~\ref{ass1} and \ref{ass2} hold. 
Then for given $m>0$ and $l\in \R$, there exists $\phi \in \mathcal{H}^1$ with $M(\phi)=m$ and $L(\phi) =l$, such that 
\[E(\phi)=e(m,l).\]
The minimizer $\phi$ is a weak solution to \eqref{statnls} with associated Lagrange multipliers $\omega, \Omega\in \R$, satisfying
\begin{equation}\label{eq:id}
e(m,l) + \frac{\lambda \sigma}{\sigma +1} \| \phi \|_{L^{2\sigma+2}}^{2\sigma+2} = \omega m + \Omega l.
\end{equation}
Moreover, in the case of zero mean vorticity $l=0$ it holds
\[
e({m, 0}) = E(\phi^\ast) = \inf _{u\in \mathcal S_m} E(u)
\]
and $\phi^\ast$ is radially symmetric. 
\end{theorem}

\begin{remark} A drawback of 
our approach is that we do not know the precise value of 
the Lagrange multiplier $\Omega$ for which this constrained minimum is achieved. Indeed, the map 
$(m,l) \mapsto (\omega,\Omega)$ is highly nonlinear and its properties are not easily accessible. 
Note however that the same is true (and generally accepted) for the map $m\mapsto \omega$, see the discussion in \cite{LeNo}.\\
\end{remark}

From now on we denote the set of bound states, i.e. energy-minimizers with prescribed mass and angular momentum by 
\[
\mathcal B_{m,l}:=\big \{ \phi\in \mathcal C_{m, l}: E(\phi)=e(m,l) \big \}\,.
\]
Recalling the change of variables \eqref{eq:vu}, Theorem \ref{thm:main} guarantees the existence of a class of stable time-dependent 
solutions $u$ to the original equation \eqref{nls}:

\begin{theorem}[Orbital stability]\label{thm:stab} 
Under the same assumptions as in Theorem \ref{thm:main}, there exist rotating, nonlinear bound state 
solutions to \eqref{nls} in the form:
\[
u_{\rm rot} (t,x) = e^{-i \omega t} e^{-it \Omega  L_z }  \phi(x),\quad \phi \in \mathcal B_{m,l}.
\]
which satisfy, for all $t\in \R$,
\[
M(u_{\rm rot}(t, \cdot)) = m\,, \quad E(u_{\rm rot}(t, \cdot)) = e(m, l)\,, \quad L(u_{\rm rot}(t, \cdot)) = l\,.
\]
Moreover, the set of these bound states is orbitally stable under the flow of \eqref{nls}.
\end{theorem}

Note that the non-uniqueness of minimizers precludes us from proving orbital stability of individual bound states. We only have stability of the full set.

\begin{remark}
An analogous orbital stability result for mass-constrained energy minimizers of \eqref{eq:mini1} was proved
in \cite{ANS}. The present paper can thus be 
seen as complimentary approach to the construction of such rotating bound states, 
with the additional property that we can prescribe their mean angular momentum. 
\end{remark}

Recall the last statement in Theorem \ref{thm:main}, which shows that for $l=0$ and $\Omega =0$ the minimization problems \eqref{eq:mini1} and \eqref{eq:mini2} are equivalent, i.e.
\[
e(m,0) = e_0(m) = E(\phi^\ast).
\]
One might wonder how \eqref{eq:mini1} and \eqref{eq:mini2}  relate to each other more generally. Indeed we shall prove the following:

 \begin{theorem}[Relation between sets of minimizers]\label{thm:connect}
 Suppose that Assumptions~\ref{ass1} and \ref{ass2} hold. Let $m>0$ and $\Omega \ge 0$ be given. Then it holds
 \[
 e_\Omega(m) = \min_{l\ge 0} \big(e(m, l) - \Omega l \big).
 \]
Denoting the set $\mathcal L_\Omega = \{ l\ge 0 \, : \,  e_\Omega(m) =  e(m, l) - \Omega l\}$, we have that
$\mathcal L_\Omega\not=\emptyset$ and 
 \[
 \mathcal G_{m,\Omega} = \bigcup_{l \in \mathcal L_\Omega} \mathcal B_{m,l}.
 \]
 \end{theorem}

As we will see in Section \ref{sec:prop}, the case $\Omega < 0$ can be treated by a simple transformation $\Omega \mapsto -\Omega$.

\begin{remark} The question of whether or not the two sets of minimizers obtained through \eqref{eq:mini1} and \eqref{eq:mini2} are in fact identical for $l\not =0$ is reminiscent of the 
distinction between minimal action ground states versus (normalized) energy ground states. This issue appears in the study of stationary 
NLS solutions with non-homogenous power-law nonlinearities, 
cf. \cite{CaSp, JL, LeNo} for an extensive discussion.
\end{remark}

As a consequence of Theorem~\ref{thm:connect} and the fact that the $\mathcal B_{m,l}$'s are mutually disjoint for different $m$'s and/or $l$'s, we can say that, for $\Omega_1,\Omega_2 \ge 0$,
\begin{equation*}
 \mathcal G_{m,\Omega_1}\cap  \mathcal G_{m,\Omega_2}=\bigcup_{l \in \mathcal L_{\Omega_1}\cap
 \mathcal L_{\Omega_2}} \mathcal B_{m,l}\,.
\end{equation*}
In particular, this means that if there exists a rotationally symmetric minimizer for some $\Omega>0$, then
$l=0\in\mathcal L_\Omega$, and so $e_\Omega(m)=e(m,0)=e_0(m)$ and
\[
\big\{\phi\in  \mathcal G_{m,\Omega}\,:\, L(\phi)=0\big\}=\mathcal B_{m,0}=\big\{\phi^*\big\}\,.
\]
Here $\phi^*$ is exactly as in Lemma~\ref{lem:zero} and Theorem~\ref{thm:main}.

The paper is now organized as follows: In Section \ref{sec:exist} we shall describe the set $\mathcal C_{m, l}$ in more detail and prove the existence of doubly constrained minimizers. 
Their orbital stability and their relationship to the minimizers of \eqref{eq:mini1} is then discussed in a series of results stated in Section \ref{sec:prop}.

%%%%%%%%%%%%%%%%%%%%%%%%%%%%%%%%%%%
%%%%%%%%%%%%%%%%%%%%%%%%%%%%%%%%%%%

\section{Existence of a doubly constrained minimizer}\label{sec:exist}

In this section we prove the existence of energy minimizers subject to two constraints. To this end, we 
first show that the constraining set $\mathcal C_{m, l}\not = \emptyset$ and thus allows for a non-trivial 
minimization procedure. 

We start by observing that in the particular case where $l$ and $m$ are such that $l  = n m$, for some $n\in \Z$, 
the constraining set $\mathcal C_{m, l}$ contains 
the sphere of radius $\sqrt{m}$ in the infinite dimensional eigenspace of $L_z$ corresponding to the eigenvalue 
$n\in \Z$. Indeed, let $(r, \vartheta, \varphi)\in [0, \infty)\times [0, \pi)\times [0, 2 \pi)$ be spherical coordinates in $\R^3$, then any function $u\in \mathcal H^1$ of the form 
\begin{equation}\label{eq:vortex}
u(r,\vartheta, \varphi) = f(r, \vartheta) e^{i n \varphi}, 
\end{equation}
satisfies both constraints, provided $f$ is a profile such that $f(0,\vartheta)=0$, and
\[
m = \| u\|^2_{L^2(\R^3)} = 2 \pi \int_0^\pi\int_0^\infty \! | f(r, \vartheta)|^2 r^2  \sin \vartheta\, dr d\vartheta.
\]
Note that here we use the convention from physics for denoting spherical coordinates, see \cite{Me}. 
However, if $\frac{l}{m}\notin\mathbb Z$, then such a $u\not \in \mathcal C_{m, l}$ and the argument fails. We treat the general case below.

\begin{remark}
Functions of the form \eqref{eq:vortex} are usually called {\it central vortex states} \cite{BC}. They are a possible ansatz for condensates in their giant vortex phase. 
In our case these vortex states are the only possible states comprising (possibly) non-regular minimizers, cf. the proof of Lemma \ref{lemreg}.
%Making such an ansatz for the minimizer $\phi$ naturally leads to yet another minimization problem for the corresponding profiles $f$, see, e.g., \cite{ArSp} for a study in this direction.
\end{remark}

\begin{lemma} \label{lem:nonempty} 
For any $m>0$ and $l\in \R$, $\mathcal C_{m,l}$ is isometrically isomorphic to a non-empty subset of
\[
\big\{(c_n)_{n\in \N_0}\subset \ell^2 \ : \ n^{1/2}c_n \in \ell^2\big\}\,.
\]
In fact, ${\rm dim} \, \mathcal C_{m,l}= \infty$.
\end{lemma}

\begin{proof} We only discuss here the case $d=3$, the situation in $d=2$ being similar. 
The operator $L_z$ is essentially self-adjoint on $\mathcal H^1$ with purely discrete spectrum $\sigma(L_z) =  \mathbb Z$. The corresponding 
eigenspaces are, in general, infinitely degenerate. To have a better sense of the multiplicities, one 
usually forms an orthonormal basis of $L^2(\R^3)$ by using the common eigenfunctions to the commuting operators $L_z$, $H$, and 
$\mathbb L^2 = (-i x\wedge \nabla )^2$, see, e.g., \cite{Me}. Any such basis element is then given by
\[
\psi_j^n(r, \vartheta, \varphi)=Y_j^n(\vartheta, \varphi) \chi_n(r),
\]
where $(r, \vartheta, \varphi)$ are spherical coordinates in $\R^3$, $\chi_n$ solves the radial Schr\"odinger equation subject to $\chi_n(0)=0$, and the $Y_j^n$ are spherical harmonics. The latter satisfy, for any $j=0, 1, \dots , \infty$:
\begin{equation}\label{eq:harm}
L_z Y_j^n(\vartheta, \varphi) = n  Y_j^n(\vartheta, \varphi), \quad n = -j, -j+1, \dots, +j.
\end{equation}
Thus, one can decompose any $\phi \in \mathcal H^1$ via
\[
\phi(r, \vartheta, \varphi)=\sum_{j=0}^{\infty} \sum_{n=-j}^j c_{n, j}\, \psi_j^n(r, \vartheta, \varphi)
=\sum_{n\in\mathbb Z}\sum_{j=|n|}^\infty c_{n, j}\, \psi_j^n(r, \vartheta, \varphi)\,,
\]
where the sequence of constants $(c_{n, j})_{n,j}\subset \ell^2$. In terms of these coefficients the mass constraint reads
	\begin{align*}
		M(\phi)&=\Big \| \sum_{j=0}^{\infty} \sum_{n=-j}^j  c_{n,j}\,\psi_j^n\Big \|_{L^2}^2 
		=\sum_{j=0}^{\infty} \sum_{n=-j}^j |c_{n,j}|^2=\sum_{n\in\mathbb Z}\sum_{j=|n|}^\infty |c_{n,j}|^2=m\,.
	\end{align*}
Similarly, one obtains, in view of \eqref{eq:harm}, that
	\begin{align*}
		L(\phi)&=\langle \phi, L_z\phi\rangle =\sum_{j=0}^{\infty} \sum_{n=-j}^j n |c_{n,j}|^2=
		\sum_{n\in\mathbb Z}\bigg(n\!\sum_{j=|n|}^\infty |c_{n,j}|^2\bigg)=l.
	\end{align*}

Let $n_1\not= n_2\in \Z$, and denote
\[
f_1 (r,\vartheta, \varphi) = \sum_{j=|n_1|}^{\infty} c_{n_1, j} \psi^{n_1}_j(r,\vartheta, \varphi), \quad f_2 (r,\vartheta, \varphi) = \sum_{j=|n_2|}^{\infty} c_{n_2, j} \psi^{n_2}_j(r,\vartheta, \varphi)\,.
\]
Then $f_1$ and $f_2$ belong to distinct eigenspaces of $L_z$, and hence $\langle f_1, f_2\rangle_{L^2}=0$ and
\[
M(f_1+f_2) =M(f_1)+M(f_2)= \sum_{j=|n_1|}^{\infty} |c_{n_1, j}|^2 + \sum_{j=|n_2|}^{\infty} |c_{n_2, j}|^2,\,.
\]
Similarly, we get
\[
L(f_1 + f_2) = n_1\sum_{j=|n_1|}^{\infty} |c_{n_1, j}|^2 + n_2\sum_{j=|n_2|}^{\infty} |c_{n_2, j}|^2=n_1 M(f_1) + n_2 M(f_2).
\]
Since $n_1 \not = n_2$, the system
\[
 \begin{pmatrix}
    1 & 1 \\
    n_1 & n_2 
  \end{pmatrix}
  \begin{pmatrix}
    M(f_1) \\
    M(f_2)
  \end{pmatrix}
  =
 \begin{pmatrix}
    m \\
    l
  \end{pmatrix}
\]
can always be solved as
\[
M(f_1)= \frac{m n_2 - l  }{n_2-n_1}, \quad M(f_2) = \frac{l - m n_1 }{n_2-n_1}.
\]
Since $m>0$, we see that the requirement $M(f_1)>0$ and $M(f_2)>0$ can always be achieved, for example if $n_2>\frac{|l|}{m}$ and $n_1< -\frac{|l|}{m}$. 

In summary $f_1+f_2 \in \mathcal C_{m,l}$. This shows that $\mathcal C_{m, l}$ contains a set which is isometrically isomorphic to the direct product of two  (infinite dimensional)
spheres in $\ell^2(\mathbb Z)$.
\end{proof}

Next, we recall the following compact embedding result proved in, e.g. \cite[Lemma 3.1]{ZJ1}:

\begin{lemma}\label{emb}
	Let $2\leq p<\frac{2d}{(d-2)_+}$, the embedding $\mathcal{H}^1\hookrightarrow L^p(\R^d)$ is compact.
\end{lemma}
Using this we can prove the existence the following minimizer. 

\begin{proposition}\label{grd}
Suppose that Assumptions \ref{ass1} and \ref{ass2} hold. 
Then for given $m>0$ and $l\in \R$, $\mathcal B_{m, l}\not = \emptyset$. More precisely, there exists a $\phi_\infty\in \mathcal C_{m,l}$,
such that 
	$$E(\phi_{\infty})=e(m,l)=\inf_{\phi \in \mathcal C_{m, l}} E(\phi).$$
%n addition, $\phi_{\infty}$ is a weak solution to \eqref{statnls} and satisfies the identity \eqref{eq:id}.
\end{proposition}

\begin{proof}
First we show that 
\[
\inf_{\phi \in \mathcal C_{m, l}} E(\phi)>-\infty.
\] 
For $\lambda>0$, this is clear since $E(u)\ge0$ for all $u\in \mathcal H^1$. In the case $\lambda\leq 0$, we apply the Gagliardo-Nirenberg inequality
\[
\|u\|_{L^{2\sigma+2}}^{2\sigma+2}\leq C\|\nabla u\|_{L^2}^{d\sigma}\|u\|_{L^2}^{\sigma(2-d)+2}, 
\]
together with Young's inequality with 
$$(p,q)=\Big(\frac{2}{d\sigma}, \frac{1}{1-d\sigma/2}\Big)$$ 
to obtain that for any $\varepsilon>0$:
$$E(u)\geq \bigg(\frac{1}{2}+\frac{C\lambda\varepsilon^p}{(\sigma+1)p}\bigg)\big\|\nabla u\big\|_{L^2}^2
+\big\|V^{1/2}u\big\|_{L^2}^2+\frac{C\lambda}{(\sigma+1)\varepsilon^qq}\big\|u\big\|^{(\sigma(2-d)+2)q}_{L^2}.$$ 
Choosing an appropriate $\varepsilon>0$ then yields the lower bound
\[
E(u)\geq \frac{1}{4}\|u\|_{\mathcal{H}^1}^2+C_{m}, 
\] 
where $$C_{m}=\frac{C\lambda}{(\sigma+1)\varepsilon^qq}m^{(\sigma(2-d)+2)q}-\frac{1}{4}m.$$

Next, in view of Lemma \ref{lem:nonempty}, there exists a minimizing sequence $(\phi_{n})_{n\in \N}\subset \mathcal C_{m,l}$. By the norm equivalence, 
    \begin{align*}
    \|\phi_{n}\|_{\mathcal{H}^1}\simeq \|\phi_{n}\|_{H^{1}(\R^d)}+\|V^{1/2}\phi_{n}\|_{L^2(\R^{d})} \ & \lesssim E(\phi_{n})+M(\phi_n) = E(\phi_{n})+m.
    \end{align*}
Since $(\phi_n)_{n \in \N}$ is a minimizing sequence of $E$, we know that $(\phi_{n})_{n\in \N}$ is a bounded sequence in $\mathcal{H}^1$. By the Banach-Alaoglu Theorem, there 
consequently exists a weakly convergent subsequence $(\phi_{n_j})_{j\in \N}$, such that
$$
\phi_{n_{j}}\rightharpoonup \phi_{\infty} \ \text{as } j\to \infty
$$
for some $\phi_\infty \in \mathcal H^1$.
The compact embedding of $\mathcal{H}^1 \hookrightarrow L^2$ 
implies that $\phi_{n_j}\to \phi_{\infty}$ strongly in $L^2$, and thus 
the mass constraint is preserved in the limit, i.e.
\begin{equation}\label{eq:masscon}
\big\|\phi_{\infty}\big\|_{L^2}^2=\lim_{j\to \infty} \big\|\phi_{n_j}\big\|_{L^2}^2=m.
\end{equation}

Next we prove that the same holds true for the angular momentum constraint $l=L(\phi_{n_j})$. 
To this end, we write
\begin{equation*}
	L(\phi_{n_j})=-i \big(\langle \phi_{n_j},x_{1}\partial_{x_2}\phi_{n_j}\rangle_{L^2}-  \langle \phi_{n_j},x_{2}\partial_{x_1}\phi_{n_j}\rangle_{L^2} \big )=:-i(A_j-B_j).
\end{equation*}
For $A_j$, we decompose for some $R>0$ (to be chosen below):
\begin{align*}
	A_j
	&=\int_{\R^d}\overline{\phi_{n_j}}\,x_{1}\partial_{x_2}\phi_{n_j}\, dx\\
	&=\int_{|x|\leq R}\overline{\phi_{n_j}}x_{1}\partial_{x_2}\phi_{n_j}\, dx+\int_{|x|>R}\overline{\phi_{n_j}}x_{1}\partial_{x_2}\phi_{n_j}\, dx\,.
\end{align*}
By Cauchy-Schwarz, we can estimate
\begin{align*}
	\Big | \int_{|x|>R}\overline{\phi_{n_j}}x_{1}\partial_{x_2}\phi_{n_j}\, dx\Big | 
	&\leq \|\nabla \phi_{n_j}\|_{L^2}\Big(\int_{|x|>R}|x|^2|\phi_{n_j}|^2\, dx\Big)^{1/2}\\
	&\lesssim\|\nabla \phi_{n_j}\|_{L^2}\Big(\int_{|x|>R}\frac{V(x)}{|x|^{k-2}}|\phi_{n_{j}}|^2\, dx\Big)^{1/2}\\
	&\leq \frac{1}{R^{\frac{k-2}{2}}}\|\nabla \phi_{n_j}\|_{L^2}\|V^{1/2}\phi_{n_j}\|_{L^2}\\
	&\lesssim \frac{1}{R^{\frac{k-2}{2}}} \|\phi_{n_j}\|_{\mathcal H^1}^2=\mathcal O \left(\frac{1}{R^{\frac{k-2}{2}}}\right).
\end{align*}
Here, the super-quadratic growth of $V$ guarantees that $k-2>0$ and hence the integral can be made arbitrarily small for $R>0$ sufficiently large. 
  
Let $\varepsilon >0$. We thus know that there exists $R_\varepsilon>0$, such that
\begin{equation}\label{eq:est1}
\forall j \in \N \ : \ \Big | \int_{|x|>R_\varepsilon}\overline{\phi_{n_j}}x_{1}\partial_{x_2}\phi_{n_j}\, dx \Big | < \frac{\varepsilon}{3}.
\end{equation}
In addition, the same holds true when we replace $\phi_{n_j}$ by $\phi_{\infty}$ in the 
above estimate.

For $|x|\le R_\varepsilon$, we know that $\phi_{n_j}\to \phi_{\infty}$ strongly in $L^2(B_{R_\varepsilon}(0))$. 
Combining this with the fact that $\partial_{x_2}\phi_{n_j}\rightharpoonup  \partial_{x_2}\phi_{\infty}$ in $L^2$ 
(which is due to the weak convergence $(\phi_{n_j})_{j\in\mathbb N}$ in $H^{1}$) implies 
that 
\[
\overline {\phi_{n_j}}x_{1}\partial_{x_2}\phi_{n_j}\to  \overline {\phi_{\infty}} x_{1}\partial_{x_2}\phi_{\infty} \quad \text{strongly in $L^1(B_{R_\varepsilon}(0))$.} 
\]
Hence, there exists some $N=N_\varepsilon\in \mathbb{N}$, such that for any $j\geq N$,
$$
\Big|\int_{|x|\leq R_\varepsilon}\overline{\phi_{n_j}}x_{1}\partial_{x_2}\phi_{n_j}-\overline{\phi_{\infty}}x_{1}\partial_{x_2}\phi_{\infty}\, dx \Big|< \frac{\varepsilon}{3}\,.
$$ 
Combining this with estimate \eqref{eq:est1} for $\phi_{n_j}$ and for $\phi_\infty$ then yields for $j\geq N_\varepsilon$ :
\begin{equation*}
\Big|\int_{\R^d}\overline{\phi_{n_j}}x_{1}\partial_{x_2}\phi_{n_j}-\overline{\phi_{\infty}}x_{1}\partial_{x_2}\phi_{\infty}\, dx\Big|< \varepsilon,
\end{equation*}
which in turn means that 
\[
\lim_{j\to \infty} \big| A_j - \langle \phi_{\infty},x_{1}\partial_{x_{2}}\phi_{\infty}\rangle_{L^2} \big| =0.
\]

For $B_j$, we can repeat the same steps. In summary this shows that the angular momentum constraint is preserved in the limit, i.e.
\[
l = \lim_{j\to \infty}L(\phi_{n_j})=-i\big(\langle \phi_{\infty},x_{1}\partial_{x_{2}}\phi_{\infty}\rangle_{L^2}  -\langle \phi_{\infty},x_{2}\partial_{x_{1}}\phi_{\infty}\rangle_{L^2}\big)= L(\phi_{\infty}).
\]
Together with \eqref{eq:masscon} this shows that $\phi_\infty \in \mathcal C_{m, l}$. 
Finally, by the weakly lower semicontinuity of $E(\phi)$, which in itself is due to the weakly lower semicontinuity of the $\mathcal H^1$-norm, we have
$$e(m,l)=\inf_{\phi \in \mathcal C_{m, l}} E(\phi) \leq E(\phi_{\infty})\leq \liminf_{j\to \infty}E(\phi_{n_j})=e(m,l).$$ 
This concludes the proof.
%A calculation of the first variation of $E$ then yields 
%\[ 
%dE(\phi_{\infty})= \omega d M(\phi_{\infty})+ \Omega dL(\phi_{\infty}),
%\] 
%for some Lagrange multipliers $\omega, \Omega \in \R$. This shows that $\phi_{\infty}\in \mathcal H^1$ is indeed a weak solution of the stationary equation \eqref{statnls}. 
%Multiplying the latter by $\overline \phi$ and integrating over $\R^d$ directly yields the identity \eqref{eq:id}.
\end{proof}

\begin{remark} 
The proof shows the necessity of having a super-quadratic potential $V$, since we require the existence of a $k>2$ to guarantee \eqref{eq:est1}, which in turn 
proves that the angular momentum constraint is preserved in the limit. At this point we do not see how 
to overcome this issue and allow for a confinement of the form $V(x) = |x|^2$. This problem might be related to the fact that the range of $l\mapsto \Omega_l$ is not known 
and thus we cannot impose any smallness assumptions on $|\Omega_l|$ as is done in the case of \eqref{eq:mini1} with quadratic potential.
\end{remark}

We now seek to derive the Lagrange condition for a given minimizer $\phi \in \mathcal C_{m,l}$. Recall that in the case of 
a minimization problem with several constraints, a minimizer $\phi$ is called {\it regular} if the derivatives of the constraining functionals at $\phi$
are {\it linearly independent}. In this case, Lagrange's theorem implies that $\phi$ satisfies the Euler-Lagrange equation 
\begin{equation}\label{Lcond}
dE(\phi) = \omega dM(\phi)+\Omega dL(\phi),
\end{equation}
with some Lagrange multipliers $\omega, \Omega\in \R$. By explicitly computing the variational derivatives,
this shows that $\phi \in \mathcal H^1$ is a weak solution of the stationary equation \eqref{statnls}. 
Multiplying the latter by $\overline \phi$ and integrating over $\R^d$ then directly yields the identity \eqref{eq:id}. 

Concerning the regularity of minimizers we shall first show:

\begin{lemma} \label{lemreg}
If $\frac{l}{m}\notin\mathbb Z$, then all minimizers are regular. In addition, regularity also holds if $\frac{l}{m}=n\in\mathbb Z$ and if the minimizer $\phi$ is not 
an eigenfunction of $L_z$ at eigenvalue $n\in \mathbb Z$. 
\end{lemma}

\begin{proof}
Let $\phi \in \mathcal C_{m,l}$ be a minimizer and assume that there exist 
$\lambda_1, \lambda_2\in\mathbb R$ not both zero, such that
\begin{equation*}\label{E:lindep}
\lambda_1 dM(\phi)+\lambda_2 dL(\phi)=0\,\,\,\Longleftrightarrow\,\,\, \lambda_1 \phi + \lambda_2 L_z \phi=0\,.
\end{equation*}
Note that $\phi$ being a minimizer implies that $M(\phi)=m>0$, and hence $\phi\not=0$. 
This in turn means that $\lambda_2\not=0$, and hence, if the constraints are linearly dependent then $\phi$ satisfies
\begin{equation*}\label{E:eig1}
L_z \phi=-\frac{\lambda_1}{\lambda_2} \phi\,.
\end{equation*}
This is the eigenvalue equation for $L_z$ whose spectrum is given by $n\in \Z$. 
Multiplying by $\overline \phi\in \mathcal C_{m,l}$ and integrating over $\R^d$ therefore yields
$$
l=-\frac{\lambda_1}{\lambda_2}\,m \,\,\,\Longrightarrow\,\,\, 
-\frac{\lambda_1}{\lambda_2}=\frac{l}{m}\, \in \Z.
$$

This shows that a minimizer $\phi\in\mathcal{C}_{m,l}$ is not regular {\it if and only if} there exists $n\in\mathbb Z$
such that $l=nm$ and $\phi \in \mathcal H^1$ is an eigenfunction of $L_z$ at eigenvalue $n$.
%\begin{equation}\label{E:eig2}
%L_z \phi=n \phi\,.
%\end{equation}
\end{proof}

%Note that $L(\phi) = n M(\phi)$ does not imply that $\phi$ is necessarily an eigenfunction of $L_z$. Hence, 
We now focus on the case of non-regular minimizers, i.e. we assume that $\frac{l}{m}=n\in\mathbb Z$ and in addition, that the minimizer, 
which we will now denote by $\phi_n$, lies within the $n$-th eigenspace
$$
\mathcal{L}_n:=\{u\in \mathcal{H}^1:\, L_z u=nu\} .
$$
Observe that any minimizer $\phi_n\in\mathcal{L}_n$ of the doubly constrained problem is also a minimizer of the 
reduced minimization problem
\begin{equation}\label{E:min_nm}
e_{n}(m):=\inf \big\{E(u)\,:\, u\in \mathcal{L}_n\,,\,M(u)=m\big\}\,.
\end{equation}
In particular, the existence result stated in Proposition \ref{grd} ensures that if the minimizer $\phi =\phi_n\in \mathcal L_n$, and thus non-regular, then $e_n(m) = E(\phi_n)$.

In order to better understand the reduced problem, we recall that $L_z=-i\partial_\varphi$ in 
cylindrical coordinates $(r,\varphi,z)\in \R^3$. Thus, $u \in L^2(\R^d)$ is an eigenfunction of $L_z$ corresponding to the eigenvalue $n$, if and only if it is given in the form of a central vortex state, i.e.
\begin{equation}\label{vortex}
u(x,y,z)=f(r,z)e^{in\varphi},
\end{equation}
with $f$ in the following weighted $L^2$-space
$$
f\in L^2_\text{cyl}(\R^2_+) := L^2\big((0,\infty)\times\mathbb R, r dr  dz\big).
$$
In particular, if $u$ is given by \eqref{vortex}, then mass constraint $M(u)  = m$ is equivalent to imposing 
$$
M_\text{cyl}(f):=2\pi\int_\mathbb R\int_0^\infty \big|f(r,z)\big|^2\,rdrdz\, = m.
$$
Furthermore, for any $u$ given by \eqref{vortex} it holds:
$$
\left[\begin{matrix}
\partial_x u \\ \partial_y u
\end{matrix}\right]=e^{in\varphi}
\left[\begin{matrix}
\cos\varphi & -\sin\varphi\\
\sin\varphi & \cos\varphi
\end{matrix}\right] \,
\left[\begin{matrix}
\partial_r f \\ \tfrac{in}{r} f
\end{matrix}\right],
$$
and hence
$$
\iiint_{\mathbb{R}^3} \big|\nabla u\big|^2\,dx\, dy\, dz= 2\pi \int_\mathbb{R} \int_0^\infty 
\left(\big|\nabla_{r,z}f\big|^2+\frac{n^2}{r^2}|f|^2\right)\,rdr dz\,.
$$
This shows that a finite energy states $u\in \mathcal H^1$ is of the form \eqref{vortex}, i.e. $u \in \mathcal L_n$, if and only if
$$
f\in \mathcal H^1_\text{\rm cyl}:=\Big\{f\in L^2_{\rm cyl}(\R^2_+)\,:\,\nabla_{r,z} f,\ \tfrac{1}{r}f,\ V^{1/2}f \in L^2_\text{cyl}(\R^2_+)\Big\}.
$$
Using this, the reduced minimization problem \eqref{E:min_nm} can be equivalently written in the form 
\begin{equation}\label{fmin}
e_n(m)=\inf\big\{E_n(f)\,:\, f\in\mathcal H^1_\text{cyl}\,,\, M_\text{cyl}(f)=m\big\}\,,
\end{equation}
where the total energy of $f\in\mathcal H^1_\text{cyl}$ is
$$
E_n(f)=2\pi\int_\mathbb R\int_0^\infty \left(\frac12 \big|\nabla_{r,z} f\big|^2+\frac{n^2}{2r^2}\,\big|f\big|^2+V\big|f\big|^2
+\frac{\lambda}{\sigma+1} \big|f\big|^{2\sigma+2}\right)\,rdrdz\, .
$$
It then follows that any minimizer $f$ of \eqref{fmin} satisfies the associated Euler-Lagrange equation, i.e.
\[
dE_{n}(f)=\omega \, dM_{\rm cyl}(f)
\]
where $\omega\in \mathbb R$ is the Lagrange multiplier corresponding to the mass constraint. Note that $\omega$ implicitly depends on the choice of $n\in \Z$ within the energy $E_n(f)$.
Explicitly, the minimizer $f$ solves the following equation
\begin{equation}\label{E:ELcyl}
    -\frac{1}{2}\nabla_{r,z}^2f+\frac{n^2}{2r^2}f+V(r,z)f+{\lambda}|f|^{2\sigma}f=\omega f\,,
\end{equation}
weakly in $\mathcal H^1_\text{cyl}$.
Multiplying by $e^{in \varphi}$ and converting back to Cartesian coordinates then shows that 
any non-regular (doubly constrained) minimizer $\phi_n \in \mathcal C_{m,nm}$ of the form 
\[
\phi_n(x,y,z)=f(r,z)e^{in\varphi}
\] 
solves
\begin{equation}\label{eq}
     -\frac{1}{2}\Delta \phi_n+V \phi _n+{\lambda}| \phi_n|^{2\sigma}\phi_n=\omega \phi_n\,,
\end{equation}
weakly in $\mathcal L_n$. 
Finally, by using the orthogonality of $\mathcal L_k$ and $\mathcal L_n$ for $k\not=n$, together with the gauge invariance of the nonlinearity, 
we infer that \eqref{eq} also holds weakly in $\mathcal H^1$. Hence, a non-regular minimizer $\phi_n\in \mathcal L_n$ 
indeed also satisfies the Lagrange condition \eqref{Lcond} but with $\Omega = 0$.

In summary, we have proved the following:
\begin{proposition}
    If $\phi_n\in \mathcal L_n$ with $n\not=0$ is a minimizer of 
    \[
    e(m, nm):=\inf\{E(u)\,:\, u\in\mathcal H^1,\, M(u)=m,\, L(u)= n m \},
    \]
    then $\phi_n$ satisfies the Lagrange condition \eqref{Lcond} with $\Omega =0$. In this case $\phi_n\in \mathcal L_n$ yields as a 
    nonlinear central vortex state solution to \eqref{eq}.
\end{proposition}

We reiterate that in our case, the existence of a minimizing profile $f$ to \eqref{fmin} follows from the 
existence of a doubly constraint minimizer $\phi \in \mathcal C_{m,l}$ with $l=nm$ (as stated in Proposition \ref{grd}), provided that 
$\phi = \phi_n \in \mathcal L_n$ and hence, non regular. By contrast, the mere existence of a minimizing profile $f$ 
(which can be proved independently, cf. \cite{ArSp}) only implies that the corresponding $\phi_n \in \mathcal L_n$ is a minimizer of the reduced 
problem \eqref{E:min_nm}, but not necessarily of the doubly constraint problem. Whether or not the doubly constraint problem indeed admits 
non-regular minimizers $\phi = \phi_n \in \mathcal L_n$ remains an open problem.

%%%%%%%%%%%%%%%%%%%%%%%%%%%%%%%%%%%
%%%%%%%%%%%%%%%%%%%%%%%%%%%%%%%%%%%

\section{Properties of minimizers}\label{sec:prop}

Having shown existence of a (doubly constrained) minimizer, we can now study some of its properties. 
First we shall prove the following orbital stability result:

\begin{proposition}\label{prop:stab}
	Let Assumptions \ref{ass1} and \ref{ass2} hold. Then the set of bound states $\mathcal B_{m,l}$ is orbitally stable under the flow of \eqref{vnls}. That is,
	for all $\varepsilon>0$ there exists $\delta=\delta(\varepsilon)>0$ such that if 
	$$\inf_{\phi\in \mathcal B_{m,l}}\|u_{0}-\phi\|_{\mathcal{H}^1}<\delta,$$
	then the solution $v\in C(\R,\mathcal{H}^1)$ to \eqref{vnls} with $v(0,x)=u_{0}\in \mathcal{H}^1$ satisfies 
	$$\sup_{t\in \R}\inf_{\phi\in \mathcal B_{m,l}}\|v(t,\cdot)-\phi\|_{\mathcal{H}^1}< \varepsilon.$$
\end{proposition}

\begin{proof}
	Suppose by contradiction that there exists a sequence 
	$(u_{0,n})_{n\in\mathbb N} \subset \mathcal{H}^{1}(\R^d)$, 
	a function $\phi_{0}\in \mathcal B_{m,l}$, a sequence of times $(t_{n})_{n\in \mathbb{N}}\subset \R$, and a constant $\varepsilon_0>0$ such that 
	\begin{equation}\label{orb}
		\lim_{n\to \infty}\|u_{0,n}-\phi_{0}\|_{\mathcal{H}^1}=0
	\end{equation}
	and 
	\begin{equation}\label{orbcon}
		\inf_{\phi\in \mathcal B_{m,l}}\|v_{n}(t_{n},\cdot)-\phi\|_{\mathcal{H}^1}> \varepsilon_0\quad\text{for all }n\geq 1\,.
	\end{equation}
Here $v_{n}\in C(\R,\mathcal{H}^1)$ is the unique global solution to \eqref{vnls} with initial data $u_{0,n}$. The strong convergence \eqref{orb} implies that
	$$\lim_{n\to \infty}M(u_{0,n})=M(\phi_{0}), \quad \lim_{n\to \infty}E(u_{0,n})=E(\phi_{0}),$$
	and the convergence of $L(u_{0,n})$ can be deduced similarly to the proof of Proposition~\ref{grd}.
	Indeed,
	\begin{align*}
		&|L(u_{0,n})-L(\phi_{0})|=\\
		&=\Big|\int_{\R^d}\!\!\!\overline{u_{0,n}}x_{1}\partial_{x_{2}}u_{0,n}-\overline{\phi_{0}}x_{1}\partial_{x_2}\phi_{0}\,dx-\int_{\R^d}\!\!\!\overline{u_{0,n}}x_{2}\partial_{x_{1}}u_{0,n}-\phi_{0}x_{2}\partial_{x_{1}}\phi_{0}\,dx\Big|\\
		&=|A_n-B_n|\leq |A_n|+|B_n|.
	\end{align*}
    For $|A_n|$, we can write
    \begin{align*}
    	|A_n|\leq&\Big|\int_{|x|\leq R}x_{1}(\overline{u_{0,n}}\partial_{x_2}u_{0,n}-\overline{\phi_{0}}\partial_{x_2}\phi_{0})\,dx\Big|\\
    	&\, + \Big|\int_{|x|>R}\overline{u_{0,n}}x_{1}\partial_{x_{2}}u_{0,n}\,dx\Big|+\Big|\int_{|x|>R}\overline{\phi_0}x_{1}\partial_{x_{2}}\phi_0\,dx\Big|.
    \end{align*}
    Now following the same steps as in the proof of Proposition \ref{grd}, we know that for any
    $\varepsilon>0$ there exists $R_\varepsilon$ sufficiently large such that
 \[
 \Big|\int_{|x|>R_\varepsilon}\overline{u_{0,n}}x_{1}\partial_{x_{2}}u_{0,n}\,dx\Big|< 
 \frac{\varepsilon}{3}\quad\text{and}\quad\Big|\int_{|x|>R_\varepsilon}\overline{\phi_0}x_{1}\partial_{x_{2}}\phi_0\,dx\Big|< \frac{\varepsilon}{3}.
 \]
The strong convergence in $\mathcal H^{1}$ implies the strong convergence in $L^2$ and $H^1$, and hence we have $\overline{u_{0,n}}\partial_{x_2}u_{0,n}\to \overline{\phi_0}\partial_{x_2}\phi_{0,n}$ strongly in $L_1$. We can then find $N_\varepsilon>0$ such that for any $n>N_\varepsilon$,
\[
\Big|\int_{|x|\leq R_\varepsilon}x_{1}(\overline{u_{0,n}}\partial_{x_2}u_{0,n}-\overline{\phi_{0}}\partial_{x_2}\phi_{0})\,dx\Big|< \frac{\varepsilon}{3}\,.
\]
Hence we have $\lim_{n\to \infty}|A_n|=0$, and similarly $\lim_{n\to\infty}|B_n|=0$, which together imply that 
\[
\lim_{n\to \infty}L(u_{0,n})=L(\phi_{0})=l\,.
\]
By mass and energy conservation,
$$\lim_{n\to \infty}M(v_{n}(t_{n}, \cdot))=M(\phi_{0}), \quad \lim_{n\to \infty}E(v_{n}(t_n, \cdot))=E(\phi_{0})\,,$$ 
and since $V$ is radially symmetric, 
$$\lim_{n\to \infty}L(v_{n}(t_n, \cdot))=L(\phi_{0})=l.$$
Furthermore, since 
\begin{align*}
\|v_{n}(t_n)\|_{\mathcal{H}^1}&\simeq \|v_{n}(t_n)\|_{L^2}+\|H^{1/2}v_{n}(t_n)\|_{L^2}\\
&\lesssim m+E(v_{n}(t_n)),
\end{align*}	
we see that $\big(v_{n}(t_n)\big)_{n\in\mathbb N}$ is a bounded sequence in $\mathcal{H}^1$. 
Hence there exists a weakly convergent subsequence $\big(v_{n_{j}}(t_{n_j})\big)_{j\in\mathbb N}$, 
such that $v_{n_{j}}(t_{n_{j}})\rightharpoonup v_{\infty}\in \mathcal{H}^1$. With the compact embedding result in Lemma \ref{emb} and following the proof of Proposition \ref{grd}, 
$$M(v_{\infty})=\|\psi_{\infty}\|_{L^2}^2=\lim_{j\to\infty}\|v_{n_{j}}(t_{n_{j}}, \cdot)\|_{L^2}^2=m,$$
as well as
$$L(v_{\infty})=\lim_{j\to \infty}L(v_{n_{j}}(t_{n_{j}}, \cdot))=l.$$
By the weakly lower semicontinuity of $E$,
$$\inf_{v\in \mathcal C_{m,l}}E(v)\leq E(v_{\infty})\leq \liminf_{j\to \infty}E(v_{n_{j}}(t_{n_{j}}, \cdot))=\inf_{v\in \mathcal C_{m,l}}E(v).$$
These show that $v_{\infty}\in \mathcal  B_{m,l}$ and $\big(v_{n_{j}}(t_{n_{j}}, \cdot)\big)_{j\in\mathbb N}$ converges strongly to $v_{\infty}$ in $\mathcal{H}^1$. Hence 
$$\inf_{\phi\in \mathcal B_{m,l}}\|v_{n_{j}}(t_{n_{j}})-\phi\|_{\mathcal{H}^1}\leq \|v_{n_{j}}(t_{n_{j}})-v_{\infty}\|_{\mathcal{H}^1}\xrightarrow{j\to \infty} 0,$$
which contradicts \eqref{orbcon}.	
\end{proof}

\begin{remark}
Note that the compact embedding $\mathcal{H}^1\hookrightarrow L^2(\R^d)$ directly implies that the mass constraint is preserved. Thus, in contrast to other results, we do not need to argue that 
$v_{n}(t_n, \cdot)$ can be renormalized to become a minimizing sequence satisfying the constraints. This is important here, since the mass $M(v)$ 
and the mean angular momentum $L(v)$ cannot be renormalized independently. 
\end{remark}

\begin{proof}[Proof of Theorem \ref{thm:stab}] The orbital stability result for solutions $v$ to \eqref{vnls} 
transfers to solutions $u$ of the original NLS \eqref{nls} via the unitary transformation 
\[
e^{-it \Omega  L_z }  :\mathcal H^1 \to \mathcal H^1, \quad v(t,x)\mapsto e^{-it \Omega  L_z }  v(x)\equiv u(t,x).
\]
Since this transformation preserves the physical conservation laws of mass, energy and angular momentum, we consequently have that the set 
\[
\{ u_{\rm rot}(t) \} \equiv \big\{ e^{-it \Omega  L_z }  \phi \, : \, \phi\in \mathcal B_{m,l}\ \text{and $t\in \R$} \big\}
\]
is orbitally stable under the flow of \eqref{nls}. 
\end{proof}

Next, we prove that if the angular momentum 
is chosen to be $l=0$, the doubly constrained minimizer is equivalent to the one obtained 
by imposing only a single mass constraint. 

\begin{lemma} \label{lem:zero} Let Assumptions \ref{ass1} and \ref{ass2} hold.
	For any $m>0$  it holds that
	$$e(m, 0) \equiv \inf_{u \in \mathcal C_{m, 0}}E(u)=\inf_{u \in \mathcal S_m}E(u).$$ 
In addition, the infimum is achieved for radially symmetric functions $\phi = \phi^*(|x|)$.	
\end{lemma}

\begin{proof}
Clearly, it is always true that 
\begin{equation}\label{eq:inf1}
\inf_{u\in \mathcal C_{m, 0}}E(u)\geq \inf_{u \in \mathcal S_m}E(u)\,.
\end{equation}  
The existence of an energy minimizer subject to a mass constrained has been proven in various contexts (and with various conditions on the potential and the nonlinearity), 
see, e.g. \cite{HHMT, HS, JL, LRY, ZJ1}. Thus, there exists $\phi^*\in \mathcal H^1$, such that 
$$
E(\phi^*)=\inf_{u\in \mathcal S_m}E(u)\,.
$$
Since $V$ is radially symmetric, a symmetric decreasing rearrangement 
implies that the minimizer $\phi^*$ can be chosen to be a radial function (cf. \cite{LL, LRY}), and hence is unique up to
phase conjugations $\phi^*\mapsto e^{i \theta}\phi^*$, $\theta \in \R$. In this case $L(\phi^*)= 0$, which in turn implies 
that $\phi^*\in\mathcal C_{m,0}$ and hence
$$
E(\phi^*)\geq \inf_{u\in \mathcal C_{m, 0}}E(u)\,.
$$ 
Combined with \eqref{eq:inf1} this yields the result. 
\end{proof}

\begin{remark}
In this proof we use the fact that $V$ is radially symmetric in order to conclude that the 
minimizer is radially symmetric too. Presumably the statement can be generalized to the case of 
merely axis-symmetric potentials $V$, satisfying $L_zV=0$. In this case, the minimizer $\phi^*$ is expected to be also axis-symmetric. Unfortunately, we could not find a reference which guarantees 
the existence (and uniqueness) of such minimizers, which is 
why we stated the result under the more restrictive condition of radially symmetric $V$.
\end{remark}

\begin{proposition}\label{prop:connect}
Let Assumptions \ref{ass1} and \ref{ass2} hold. Then we have the following properties:
\begin{itemize}
\item[(i)] For any $m>0$ and $l,\Omega\in\mathbb R$
\begin{equation}\label{eq:e_symm}
e(m,l)=e(m,-l)  \ge e_\Omega(m) +  |\Omega l |\,.
\end{equation}
\item[(ii)] Given $\Omega\in\mathbb R$ and $m>0$, let $\phi_\Omega\in\mathcal S_m$ be a minimizer such that
$E_\Omega(\phi_\Omega) = e_\Omega(m)$, and let $l_\Omega=L(\phi_\Omega)$. Then $E(\phi_\Omega)=e(m,l_\Omega)$, 
and the minima satisfy
\begin{equation}\label{eq:connect}
e(m,l_\Omega)=e_\Omega(m)+\Omega l_\Omega \quad\text{and}\quad \Omega l_\Omega\ge 0\,.
\end{equation}
\end{itemize}
\end{proposition}

Note that if one takes $\Omega=0$ in (ii), we already know that the minimizer is radially symmetric and hence $l_\Omega = 0$ in this case. 
In turn this implies that \eqref{eq:connect} simplifies to the statement of Lemma \ref{lem:zero}.

\begin{proof}
First consider the transformation $\mathcal{H}^1 \ni u\mapsto \tilde u\in \mathcal{H}^1$ given by 
\begin{equation*}
\tilde u(x_1,x_2,z)=u(-x_1,x_2,z)\,,
\end{equation*}
and observe that
\begin{equation*}
M(\tilde u)=M(u)\,,\quad E(\tilde u)=E(u)\,,\quad L(\tilde u)=-L(u)\,,\quad\text{and}\quad E_\Omega(\tilde u)=E_{-\Omega}(u)\,.
\end{equation*}
These properties directly imply that $u\in\mathcal C_{m,l}$ if and only if 
$\tilde u\in \mathcal C_{m,-l}$ and so
\begin{equation}\label{eq:emlsymm}
e(m,l)=e(m,-l)\quad\text{for any } m>0\text{ and }l\in\R\,.
\end{equation}
Now, let $m>0$, $l,\Omega\in\R$, and note that 
\[
E_{\Omega} (u) = E(u) - \Omega l, \quad\text{for all $u\in \mathcal C_{m,l}$.}
\]
Since $\mathcal C_{m,l}\subset \mathcal S_m$, we find that
\[
\inf_{u \in \mathcal C_{m, l}} E_\Omega(u) = \inf_{u \in \mathcal C_{m, l}} E(u) - \Omega l \ge \inf_{u \in \mathcal S_m} E_\Omega(u)\,,
\]
and hence $e(m,l)  \ge e_\Omega(m) +  \Omega l$.
Combining this with \eqref{eq:emlsymm} yields the identity \eqref{eq:e_symm}.

Now we turn to (ii), and let $\phi_\Omega \in \mathcal S_m$ be such that
$E_\Omega(\phi_\Omega) = e_\Omega(m)$. The existence of such a minimizer was proved in \cite{RS1, RS2}. Denoting
$L(\phi_\Omega) = l_ \Omega$ yields $\phi_\Omega\in \mathcal C_{m,l_\Omega}$. 
In addition, for any $u\in \mathcal C_{m,l_\Omega}\subset \mathcal S_m$,
\begin{equation*}
E(u)= E_\Omega (u) + \Omega L(u)=E_\Omega (u) + \Omega l_\Omega
\geq e_\Omega(m)+ \Omega l_\Omega = E(\phi_\Omega)\,,
\end{equation*}
which implies that $\phi_\Omega$ is also a minimizer of $E$ on $C_{m,l_\Omega}$, as claimed.
This in turn shows that
\[
e(m,l_\Omega)=E(\phi_\Omega)=E_\Omega(\phi_\Omega)+\Omega L(\phi_\Omega)=e_\Omega(m)
+\Omega l_\Omega\,.
\]
Combining this identity with \eqref{eq:e_symm} then implies that $\Omega l_\Omega\ge0$,
which completes the proof.
\end{proof}

The transformation $u\mapsto\tilde u$ and the associated relations between their energies,
in particular the fact that $E_\Omega(\tilde u)=E_{-\Omega}(u)$,
allows us to restrict, without loss of generality, the statement of Theorem~\ref{thm:connect} to $\Omega\ge 0$.

Proposition~\ref{prop:connect} is the basis for the following:
\begin{proof}[Proof of Theorem \ref{thm:connect}]
Let $\Omega>0$. From Proposition \ref{prop:connect}(i) we have the lower bound $e(m,l)-\Omega l \ge e_\Omega(m)$ for all $l\ge 0$. 
In addition, item (ii) shows that there exists $l_\Omega\ge 0$, such that
\[
e(m,l_\Omega)-\Omega l_\Omega =  e_\Omega(m) .
\]
Together these imply that 
\[
 e_\Omega(m) = \min_{l\ge 0} \big(e(m, l) - \Omega l \big),
 \]
 as claimed. Furthermore, Proposition \ref{prop:connect}(ii) also implies that 
 \[
 \mathcal G_{m, \Omega} \subset \bigcup_{l \in \mathcal L_\Omega} \mathcal B_{m,l}.
 \]
 Conversely, let $l_*\in \mathcal L_\Omega= \{ l\ge 0 \, : \,  e_\Omega(m) =  e(m, l) - \Omega l\}$ and $\phi \in \mathcal B_{m, l_*}$. Then
 \[
 E_\Omega(\phi) = E(\phi) - \Omega l_* = e(m, l_*) - \Omega l_* = e_\Omega(m),
 \]
 where the last identity follows from the definition of $\mathcal L_\Omega $. 
 In turn this implies that $\phi\in  \mathcal G_{m, \Omega}$, which completes the proof.
 \end{proof}

We close with one last observation. One of the strengths
of Theorem~\ref{thm:connect} is the fact that, for a given
$\Omega>0$, the definition of $\mathcal L_\Omega$ does not
require finding a specific minimizer of either $E_\Omega$ or $E$.
However, we also do not know the range of mean angular
momenta that emerge for minimizers of $E_\Omega$, i.e. we do not know
what the following set is:
\[
\bigcup_{\Omega\ge0} \mathcal L_\Omega\subset [0,\infty)\,.
\]
In other words, we cannot (yet) say which values $l\in\mathbb R$ appear
as  mean angular momenta of minimizers $\phi_\Omega$ of $E_\Omega$ in
$\mathcal S_m$.

Another way to look at this is the following: it is true
that given $m>0$ and $l\ge0$, we have
\[
e(m,l)\ge\sup_{\Omega\ge0} \big(e_\Omega(m)+\Omega l\big)\,.
\]
But we do not currently have a way to ensure
that there exists $l_\Omega$
such that $e(m,l_\Omega)=e_\Omega(m)+\Omega l_\Omega$, since
a minimizer $\phi_l\in\mathcal C_{m,l}$ of $E$ is a critical
point in $\mathcal S_m$, but not necessarily a minimizer,
of $E_{\Omega_l}$ (with  
$\Omega_l$ the Lagrange multiplier associated to $\phi_l$ and $E$).

%%%%%%%%%%%%%%%%%%%%%%%%%%%%%%%%%%%
%%%%%%%%%%%%%%%%%%%%%%%%%%%%%%%%%%%

\bibliographystyle{amsplain}

\begin{thebibliography}{99}

\bibitem{AMS} P. Antonelli, D. Marahrens, C. Sparber,  {On the Cauchy problem for nonlinear Schr\"odinger equations with rotation}. {\it Discrete Contin. Dyn. Syst.}, {\bf 32} (2012), no. 3, 703--715.

\bibitem{ANS} J. Arbunich, I. Nenciu, and C. Sparber, {Stability and instability properties of rotating Bose-Einstein condensates}. {\it Lett. Math. Phys.}, {\bf 109} (2019) no. 6, 1415--1432.

\bibitem{ArSp} A.~K. Arora and C. Sparber, {Self-bound vortex states in nonlinear Schr\"odinger equations with LHY correction}. {\it NoDEA Nonlinear Differential Equations Appl.} {\bf 30} (2023), no. 1, Paper No. 14, 25 pp.

\bibitem{BC} W. Bao and Y. Cai, {Mathematical theory and numerical methods for Bose-Einstein condensation}. {\it Kinetic Related Models}, {\bf 6} (2013), no. 1, 1--135.

\bibitem{BQZ} W. Bao, Q. Du, and Y. Z. Zhang, {Dynamics of rotating Bose-Einstein condensates and its efficient and accurate numerical computation}. {\it SIAM J. Appl. Math.}, {\bf 66} (2006), 758--786.

\bibitem{BCPY} J. B. Bru, M. Correggi, P. Pickl, and J. Yngvason, {The TF limit for rapidly rotating Bose gases in anharmonic traps}. {\it Comm. Math. Phys.}, {\bf 280} (2008), 517--544.

\bibitem{Ca} R. Carles, {Sharp weights in the Cauchy problem for nonlinear Schr\"odinger equations with potential}. {\it Z. Angew. Math. Phys.}, {\bf 66} (2015), no. 4, 2087--2094.

\bibitem{CaSp} R. Carles and C. Sparber, {Orbital stability vs. scattering in the cubic-quintic Schr\"odinger equation}. {\it Rev. Math. Phys.}, {\bf 33} (2021), pp. 2150004.

\bibitem{Co} N. R. Cooper, {Rapidly rotating atomic gases}. {\it Advances Phys.}, {\bf 57} (2008), 539--616.

\bibitem{CPRY} M. Correggi, F. Pinsker, N. Rougerie, and J. Yngvason, {Giant vortex phase transition in rapidly rotating trapped Bose-Einstein condensates}. 
{\it Eur. Phys. J. Spec. Top.}, {\bf 217} (2013), 183--188.

\bibitem{CRY} M. Correggi, T. Rindler-Daller, and J. Yngvason, {Rapidly rotating Bose-Einstein condensates in homogeneous traps}. {\it J. Math. Phys.}, {\bf 48} (2007), no. 10, 102103, 17pp.

\bibitem{DaSt} F. Dalfovo and S. Stringari, {Bosons in anisotropic traps: Ground state and vortices}. {\it Phys. Rev. A}, {\bf 53} (1996), 2477--2485.

\bibitem{Fe} A. Fetter, {Rotating trapped Bose-Einstein condensates}. {\it Rev. Mod. Phys.}, {\bf 81} (2009), pp. 647.

\bibitem{HHMT} F. Hadj Selem, H. Hajaiej, P. A. Markowich, and S. Trabelsi, {Variational Approach to the Orbital Stability of Standing Waves of the Gross–Pitaevskii Equation}. 
{\it Milan J. Math.}, {\bf 82} (2014), 273--295.

\bibitem{HS} H. Hajaiej and C.~A. Stuart, {On the variational approach to the stability of standing waves
for the nonlinear Schr\"odinger equation}. {\it Adv. Nonlinear Stud.}, {\bf 4} (2004), 469--501.

\bibitem{JL} L. Jeanjean and S.~S. Lu, {On global minimizers for a mass constrained problem}. {\it Calc. Var. Partial Differential Equ.}, {\bf 61} (2022), pp. 214.

\bibitem{LeNo} M. Lewin and S. Rota Nodari, {The double-power nonlinear Schr\"odinger equation and its generalizations: uniqueness, non-degeneracy and applications}. 
{\it Calc. Var. Partial Differential Equ.,} {\bf 59} (2020), pp. 197.

\bibitem{LL} E. Lieb and M. Loss, {Analysis}. Graduate Studies in Mathematics Vol. 14. American Math. Soc., Providence, 1997.

\bibitem{LRY} E.H. Lieb, R. Seiringer, and J. Yngvason, {Bosons in a Trap: A Rigorous Derivation of the Gross-Pitaevskii Energy Functional}. {\it Phys. Rev. A}, {\bf 61} (2000), 043602, 13pp.

\bibitem{Me} A. Messiah, {Quantum Mechanics}. Dover Publ. Inc., 1995.

\bibitem{RS1}  R. Seiringer, {Gross-Pitaevskii theory of the rotating gas}. {\it Comm. Math. Phys.}, {\bf 229} (2002), 491--509.

\bibitem{RS2}  R. Seiringer, {Ground state asymptotics of a dilute, rotating gas}. {\it J. Phys. A: Math. Gen.}, {\bf 36} (2003), 9755--9778.

\bibitem{YZ} K. Yajima, and G. Zhang, {Local smoothing property and Strichartz inequality for Schr\"odinger equations with potentials
superquadratic at infinity}. {\it J. Differential Equ.}, {\bf 202} (2004), 81--110.

\bibitem{ZJ1}  J. Zhang, {Stability of standing waves for nonlinear Schr\"odinger equations with unbounded potentials}.
{\it Z. Angew. Math. Phys.}, {\bf 51} (2000), no. 3, 498--503.





\end{thebibliography}

\end{document}